\documentclass{article}
\usepackage[utf8]{inputenc}

\usepackage{amsfonts}
\usepackage{amsmath}
\usepackage{amssymb}
\usepackage{amsthm}
\usepackage{stmaryrd}
\usepackage{calrsfs}
\DeclareMathAlphabet{\pazocal}{OMS}{zplm}{m}{n}

\usepackage[backend=biber, style=authoryear]{biblatex}
\usepackage{blkarray}
\usepackage{mathtools}
\usepackage{paralist}
\usepackage{tikz}
\usepackage{tikz-cd}

\usepackage{csquotes}
\usepackage{enumitem}

\usepackage{hyperref}

\newtheorem{theorem}{Theorem}[]
\newtheorem{corollary}{Corollary}[]
\newtheorem{lemma}{Lemma}[]
\newtheorem{proposition}{Proposition}[]
\theoremstyle{definition}
\newtheorem{definition}{Definition}[]
\theoremstyle{remark}

\title{At the Edge of Putnam's Program: Limitative Results For Computable Inductive Logics}
\author{Antoine Mercier, Josiah Lopez-Wild, and Elijah Spiegel}
\date{Last updated \today}

\begin{document}

\maketitle

\begin{abstract}
    The task of inductive logic is to develop a formal framework to analyze inductive reasoning.
    Historically this was accomplished by assigning probabilities to sentences of a logical language.
    Two natural criteria for such a system are: (i) the underlying language should be rich enough to express scientific hypotheses, and (ii) the probabilities should be, in some sense, accessible.
    The first criterion suggests that the language should at least contain the language of arithmetic, while the second suggests that probabilities should be computable.
    We show that these two criteria are in tension with one another.
    Various natural proposals for an inductive logic result in probabilities that are not arithmetically definable, much less computable.
    We isolate the assumptions responsible for this result, and search for a weaker inductive logic with more accessible probabilities.
    The most natural weakening results in probabilities that are arithmetically definable but still are not computable.
\end{abstract}

\section{Introduction}

The task of inductive logic is to develop a formal framework to analyze inductive reasoning. We do so by assigning probabilities to sentences of a logical language. In simple examples involving coin flips and colored balls this is easily done and the framework ends up resembling elementary probability theory. But as we tackle more sophisticated problems that emerge in scientific contexts, we quickly need the resources of higher mathematics. A more developed inductive logic requires a rich language that can express mathematically structured hypotheses and empirical observations.

In a series of papers, Putnam (\cite{Putnam1956-PUTADO, Putnam_1979}) described two reasonable constraints on the project of developing an inductive logic: 
\begin{itemize}
    \item [(1)] The underlying language should be rich enough to be able to express scientific hypotheses.
    \item[(2)] The probabilities over the language should be in some sense accessible. 
\end{itemize}

The first condition is a matter of the goal of inductive logic. Putnam posited as a ``final goal'' the ``definition of DC [degree of confirmation] for a language rich enough for the formalization of empirical science as a whole.''\footnote{\cite[p. 271]{Putnam_1979}} The second condition is a matter of our means for attaining that goal. Putnam says we can ``think of a system of inductive logic as a design for a `learning machine': that is to say, a design for a \textit{computing machine} that can extrapolate certain kinds of empirical regularities from the data.''\footnote{\cite[p. 297]{Putnam_1979}, emphasis added.} 
Putnam's first condition assures us that our inductive logic can model and analyze the actual practice of science.\footnote{A similar idea has been suggested by \cite{Romeijn2009-ROMILA} although in a slightly different context.} The second condition presses the point that the probabilities over such a language should not be highly unconstructive functions but should be the kinds of mathematical objects that we can use in our reasoning.\footnote{This idea can be found in the work \cite{Sterkenburg2019-STEPDA-11}, and in the project of computable Bayesianism as in \cite{HWZB2026Bounded}, \cite{ZafforaBlando2025}, \cite{ZafforaBlando2024Pride}, \cite{LopezWild2025vNM}, and \cite{pittphilsci23862}, among others.}
In this paper, we explore the subtleties involved in setting up an inductive logic which is amenable to accessible probabilities over rich languages. 
We keep our formalization of `accessible' general by considering complexity classes of definable probabilities. In this way, we can talk about more or less accessible priors where those belonging to the lower complexity classes are seen as more accessible, with computability as the gold standard.

The work of \textcite{G&S1982-GAIPOR} marks an important milestone in the development of an inductive logic satisfying conditions (1) and (2). The authors demonstrate that once our underlying language is capable of interpreting the theory of arithmetic, it becomes strong enough to define probabilities over (the Gödel codes of) its sentences, putting it at risk of bumping into the limitative results of Gödel and Tarski. This demonstrates a tension between conditions (1) and (2): the more expressive our language is, the more complex we need our probabilities over this language to be in order to avoid contradiction with Gödel and Tarski's theorems. 

Gaifman has argued that this tradeoff between the expressivity of the language and the complexity of the prior -- which he credits Putnam as having first identified -- is an important and neglected problem in the philosophy of probability.\footnote{\cite{Hendricks2009-HENPAS-5}.} For reasons we will explain shortly, we believe that there is still no fully satisfying framework in the literature to study computable priors over arithmetic languages. The Gaifman and Snir framework comes close and yields many results which provide important insights into this complexity tradeoff, but their framework and those like it face problems in the way definable probabilities are defined.

% The goal of this paper is to demonstrate that assumptions ubiquitously baked-in to the setup of these inductive logics directly clash with conditions (1) and (2). Indeed, the most common ways of defining a probability over a logical language involve some basic assumptions about the domains of the models and a condition on how probability functions should behave over quantified sentences. As we will demonstrate, these assumptions block inductive logics from admitting probabilities which are definable anywhere on the arithmetic complexity hierarchy. While these assumptions are very well-motivated by the intuitions which have guided the formalization of inductive logics, rethinking the setup of these logics in light of conditions (1) and (2) reveals that these assumptions amount to defining a logical framework around a standard interpretation of arithmetic. The results of Gödel, Tarski, Löwenheim, and Skolem, have shown us that such standard interpretations are not definable in first-order languages. 

The structure of the paper is as follows. In \S 2, we set out notation and conventions. In \S 3, we show that existing frameworks like Gaifman and Snir's make fundamental assumptions which block them from satisfying Putnam's conditions. In \S 4, we show how the problem emerges from a fixed domain assumption and what's know as the Gaifman condition which are both common to many frameworks for inductive logic. This motivates the framework we present in \S 5, which is defined in terms of the provability relation. We show that this framework fares better in the complexity tradeoff by admitting arithmetically definable priors of relatively low complexity. However, we show that this is still insufficient to achieve our main goal of finding a computable prior over an arithmetic language. 

% In order to allow for arithmetically definable priors -- and hence, open the possibility that our prior is `of some use to somebody' -- we consider a framework for inductive logic where we define probabilities in terms of the provability relation which we know to be arithmetically ($\Sigma_1$) definable. We will then show that this framework fairs better in the complexity tradeoff by admitting arithmetically definable priors of relatively low complexity. It will then be shown, however, that dropping these assumptions is still insufficient to achieve our main goal of finding a computable prior over an arithmetic language. We will then isolate the problem of logical omniscience as the final blockade for developing a theory of computable probabilities over arithmetic languages.
% Revisiting the set up of an inductive logic from the point of view of Putnam's conditions provides a new lens to analyze the assumptions made in our inductive frameworks as well as provides us with a new motivation for thinking about the old problem of logical omniscience.

\section{Notation and Conventions}

In this paper we are concerned with probabilities defined on sentences in the language $\pazocal{L}_A$ of arithmetic, or some extension thereof.
$\pazocal{L}_A$ is a first-order language augmented with the symbols $S, +, \times, x^y, <$, which are intended to denote the successor function, addition, multiplication, exponentiation, and strict inequality, respectively.
We assume a countable set of variables $x, y, z, x_1, x_2, \ldots$ that bind to the first-order quantifiers $\forall x$, $\exists x$.
We define the set of formulae $\pazocal{FL}_A$ of $\pazocal{L}_A$ inductively in the usual manner; in particular we let $\pazocal{SL}_A$ be the set of sentences of $\pazocal{L}_A$, i.e., the set of formulae with no free variables. 
We also define a set of \textit{numeral constants}: letting $S^n(0)$ be shorthand for $n$ applications of the successor function to $0$, we define the numerals $\underline{n}:=S^n(0)$ for all $n \in \omega$.
We let $\mathsf{PA}$ denote the first-order theory of Peano arithmetic.
We also let $\omega$ denote the first countable ordinal and let $\mathbb{N}:= \langle \omega, \underline{0}, S, +, \times \rangle$ denote the standard model of the natural numbers.
We fix a Gödel numbering $\ulcorner \cdot \urcorner: \pazocal{FL}_A \to \omega$.

As is customary, we define \textit{bounded quantifiers} as follows: if $\varphi(x)$ is a formula of $\pazocal{L}_A$ with $x$ a free variable, then $\exists x < n \,\varphi := \exists x (x < n \land \varphi(x))$ and $\forall x < n \,\varphi := \forall x(x < n \rightarrow \varphi(x))$.
We can then define the usual \textit{arithmetical hierarchy of formulae} by induction on quantifier blocks.
We let $\Delta_0 = \Sigma_0 = \Pi_0$ denote the set of formulae containing only bounded quantifiers.
A formula $\varphi$ is $\Sigma_{n+1}$ if it is of the form $\exists x_1\ldots \exists x_m \, \psi(x_1, \ldots, x_m)$, where $\psi$ is $\Pi^0_{n}$.
A formula $\varphi$ is $\Pi^0_{n+1}$ if it is of the form $\forall x_1 \ldots \forall x_m \, \psi(x_1, \ldots, x_m)$, where $\psi$ is $\Sigma^0_n$.
A formula $\varphi$ is $\Delta^0_n$ if there is a $\Sigma^0_n$ formula $\psi_0$ and a $\Pi^0_n$ formula $\psi_1$ such that $\mathsf{PA} \vdash \varphi \leftrightarrow \psi_0 \leftrightarrow \psi_1$.
As usual, because of Post's Theorem (Soare Theorem 4.2.2), the $\Delta_1$-, $\Sigma_1$-, and $\Pi_1$-definable sets of naturals are precisely the computable, computably enumerable (c.e.), and co-computably enumerable (co-c.e.) sets, respectively.

We define the rational number system $\mathbb{Q}$ in $\pazocal{L}_A$ in the standard way.
This allows us to define real-valued functions in $\pazocal{L}_A$, as follows.
Suppose $g: \omega^{k+1} \to \mathbb{Q}$ is a rational-valued function defined by the $\pazocal{L}_A$-formula $\varphi$, i.e., $\mathsf{PA}$ proves that for all $x_1, \ldots, x_{k+1}$ there is a unique $y$ such that
\[g(x_1, \ldots, x_{k+1}) = y \leftrightarrow \varphi(x_1, \ldots, x_{k+1},y).\]
We define a real-valued function $f: \omega^k \to \mathbb{R}$ via approximations.
We say that $g:\omega^{k+1} \to \mathbb{Q}$ is an \textit{approximation for} $f$ if
    \begin{compactitem}
        \item[-] $dom(g) = dom(f) \times \omega$; and
        \item[-] $|f(x_1, \ldots, x_k) - g(x_1, \ldots, x_k, n)| \leq 2^{-n}$ for all $x_1, \ldots, x_k \in dom(f)$ and $n \in \omega$. 
    \end{compactitem}
In this case we say that $f$ is defined by the formula $\varphi$.
Thus we say that $f: \omega^k \to \mathbb{R}$ is $\Delta_n$ (resp. $\Sigma_n$, $\Pi_n$) if and only if $f$ is defined by a $\Delta_n$ (resp. $\Sigma_n$, $\Pi_n$) formula $\varphi$.
In particular we say that $f$ is \textit{arithmetically definable} if there is a $\pazocal{L}_A$-formula $\varphi$ such that $f$ is defined by $\varphi$.
Finally, if $\psi \in \pazocal{SL}_A$, we say that a set (resp. relation, function) is $\Delta^\psi_n$-, $\Sigma^\psi_n$-, or $\Pi^\psi_n$-\textit{definable} if the set (resp. relation, function) is $\Delta_n$-, $\Sigma_n$-, or $\Pi_n$-definable with (the set defined by) $\psi$ as a parameter.

\section{Probabilities Over Rich Languages: The Gaifman-Snir Framework}

In this section, we introduce the framework of \textcite{G&S1982-GAIPOR}, which marks the first major advance in developing an inductive logic satisfying Putnam's conditions. While the authors make significant headway in understanding the tradeoff between the sophistication of a logical language and the complexity bounds on the priors definable over such a language, they make certain framework decisions in defining probability functions over logical languages which we will show to be in conflict with Putnam's conditions. In this section we introduce the Gaifman and Snir framework and demonstrate how it clashes with conditions (1) and (2). Although the worries we raise are not lost on Gaifman and Snir, we will argue that their proposed solution is not faithful to condition (1).

The underlying idea in \cite{G&S1982-GAIPOR} is to define `empirical predicates' over a base language capable of expressing basic arithmetical statements.\footnote{It is interesting to point out here that Carnap himself suggested using the language of arithmetic as a base language in a very small section of \cite[pp. 62-65]{Carnap1950-CARLFO}.} A crucial element of this setup is that the symbols in $\pazocal{L}_A$ are all interpreted in the same standard way. That is, Gaifman and Snir assume that all formulae of $\pazocal{L}_A$ are interpreted in the standard model $\mathbb{N}$. As we will see shortly, this decision has some unpleasant consequences for probabilities over extensions of $\pazocal{L}_A$. Nevertheless the intuitive idea seems to be this: we should assign probability $1$ to the truths of arithmetic since the focus is really on our credences about the empirical predicates; the arithmetic is just in the background to help express more complex statements about these empirical predicates.

An extension $\pazocal{L}_A^+ \supseteq \pazocal{L}_A$ is given by introducing finitely many `empirical symbols' to the base language. These can come in the form of predicates or functions of any finite arity. Models of $\pazocal{L}_A^+$ all share the same interpretation of the `mathematical' symbols of $\pazocal{L}_A$ but differ in their interpretations of the added empirical symbols. In particular, Gaifman and Snir fix the domain of every model to be the standard natural numbers.
Where $Mod(\pazocal{L}_A^+)$ denotes the set of models with the fixed interpretation of $\pazocal{L}_A$, we write $Mod(\pazocal{L}_A^+) \vDash \varphi$ to assert that $\pazocal{M} \vDash \varphi$ for all $\pazocal{M} \in Mod(\pazocal{L}_A^+)$ and sentences $\varphi \in \pazocal{L}_A^+$. In this way, every model in $Mod(\pazocal{L}_A^+)$ shares the same domain -- namely, $\omega$ -- and hence all empirical predicates are interpreted over the standard numbers.

Probabilities are defined as functions over the \textit{sentences} of $\pazocal{L}_A^+$, denoted $\pazocal{SL}_A^+$, as follows:

\begin{definition}\label{def: Gaifman Snir probability function}
    A \textit{Gaifman-Snir probability function} on $\pazocal{L}^+_A$ is a function $P_\mathbb{N}: \pazocal{SL}^+_A \to [0,1]$ satisfying the following axioms:
    \begin{enumerate}[label=(G\arabic*)]
        \item If $Mod(\pazocal{L}^+_A) \vDash \varphi$ then $P_\mathbb{N}(\varphi)=1$.
        \item If $Mod(\pazocal{L}^+_A) \vDash \neg (\varphi \land \psi)$ then $P_\mathbb{N}(\varphi \lor \psi )=P_\mathbb{N}(\varphi) + P_\mathbb{N}(\psi)$.
        \item $P_\mathbb{N}\big(\exists x \psi(x) \big) = \sup\left\{P_\mathbb{N}\left( \bigvee_{i<n} \psi(\underline{i})\right) \mid n \in \omega \right\}.$
    \end{enumerate}
\end{definition}

\noindent
Informally, (G1) says that probability 1 is given to all logical and arithmetic truths, (G2) assures us that our probability is finitely additive\footnote{Since we cannot formulate infinite disjunctions in our language the debate over finite or infinite additivity is not relevant. Condition (G2) is by no means taking a stance on this issue.}, and (G3), often referred to as `Gaifman's condition',\footnote{This was first introduced in \cite{gaifman1964concerning} and then later referred to as `Gaifman's condition' by \cite{scott1966assigning}.} makes sure that probability functions are uniquely determined by their behavior on quantifier-free sentences. 
On the basis of this definition, one can prove that the probability behaves in an intuitive way on the logical operations:

\begin{proposition}[\cite{Paris2011-PARPIL-3}]
    Let $P_\mathbb{N} : \pazocal{SL}^+_A \to [0,1]$ be a Gaifman-Snir probability function. Then
    \begin{compactitem}
        \item[(1)] $P_\mathbb{N}(\neg \varphi) = 1 - P_\mathbb{N}(\varphi)$
        \item[(2)] If $Mod(\pazocal{L}_A^+) \vDash \neg \varphi$, then $P_\mathbb{N}(\varphi)=0$
        \item[(3)] If $Mod(\pazocal{L}_A^+)\vDash \varphi \to \psi$, then $P_\mathbb{N}(\varphi) \leq P_\mathbb{N}(\psi)$
        \item[(4)] If $Mod(\pazocal{L}_A^+) \vDash \varphi \leftrightarrow \psi$, then $P_\mathbb{N}(\varphi) = P_\mathbb{N}(\psi)$
        \item[(5)] $P_\mathbb{N}(\exists x \psi(x)) = \lim\limits_{n \to \infty} P_\mathbb{N}\big(\bigvee_{i<n} \psi(\underline{i})\big)$
        \item[(6)] $P_\mathbb{N}(\forall x \psi(x)) = \lim\limits_{n \to \infty} P_\mathbb{N}\big(\bigwedge_{i<n} \psi(\underline{i})\big)$ 
    \end{compactitem}
\end{proposition}

While statements (5) and (6) reinforce the role Gaifman's condition plays in assuring that the probabilities are defined uniquely by their behavior on the quantifier-free sentences, it should be stressed that this state of affairs is only made possible because the set of individual constants $\underline{1}, \underline{2}, \underline{3}, \ldots$ in $\pazocal{L}_A^+$ exhausts the domain of every model.

We now turn to bringing out some concerning aspects of the Gaifman-Snir framework. Recall that we formalize the idea of an ``accessible prior'' in terms of arithmetical complexity. So, for a prior to be in any way accessible, it should be arithmetically definable. But this is a fairly serious problem with the Gaifman and Snir framework: Gaifman-Snir probability functions are not arithmetically definable.

\begin{theorem}\label{thm:No Definable GS}
    There are no arithmetically definable Gaifman-Snir probability functions on $\pazocal{L}^+_A$.
\end{theorem}

\noindent
The proof of this theorem relies on Tarski's ``undefinability of truth'' theorem, which we state here for completeness.

\begin{theorem}[Tarski's Undefinability of Truth]\label{thm:Tarski}
    Let $\pazocal{M} \vDash \mathsf{PA}$. There is no $\pazocal{L}_A$-formula $\theta(x)$ such that for all $\pazocal{L}_A$-sentences $\varphi$,
    \[\pazocal{M} \vDash \theta(\ulcorner\varphi\urcorner) \; \text{ iff } \; \pazocal{M} \vDash \varphi.\]
\end{theorem}

Since the Gaifman-Snir framework fixes the interpretation of formulae of $\pazocal{L}_A$ in terms of the standard model $\mathbb{N}$, the condition (G1) forces any arithmetically definable probability function on $\pazocal{L}_A^+$ to violate Tarski's theorem. This is the idea behind the proof of Theorem \ref{thm:No Definable GS}.

\begin{proof}[Proof of Theorem \ref{thm:No Definable GS}]
    Suppose towards a contradiction that there is an arithmetically definable Gaifman-Snir probability $P_\mathbb{N}$ over $\pazocal{L}_A^+$. In particular, let $g_{P_\mathbb{N}}: \omega \times \omega \to \mathbb{Q}$ be the arithmetically definable function that approximates $\ulcorner P_\mathbb{N} \urcorner$ where $\ulcorner P_\mathbb{N}\urcorner : \omega \to \omega$ defined by $\ulcorner P_\mathbb{N}\urcorner (\ulcorner \varphi\urcorner) = P_\mathbb{N}(\varphi)$ for all $\varphi \in \pazocal{SL}_A^+$. Let $\text{Sent}(x)$ be the formula expressing that $x$ is the Gödel code for a sentence in the language of arithmetic. Then the set
    \[T_\mathbb{N}:=\{\ulcorner \varphi \urcorner \mid \text{Sent}(\ulcorner \varphi \urcorner) \wedge \forall n(|1 - g_{P_\mathbb{N}}(\ulcorner\varphi \urcorner, n)|\leq 2^{-n}) \}\]
    is arithmetically definable. But for all $\varphi \in \pazocal{SL}_A$, $\mathbb{N} \vDash \varphi$ iff $P_\mathbb{N}(\varphi)=1$, and so 
    \[T_\mathbb{N} = \{\ulcorner \varphi \urcorner \mid \varphi \in \pazocal{SL}_A \wedge \mathbb{N} \vDash \varphi\},\]
    which is not arithmetically definable by Tarski's theorem (\ref{thm:Tarski}). Contradiction.
\end{proof}

This feature of their system is not lost on Gaifman and Snir. Indeed they essentially state theorem 1 without proof.\footnote{See \cite[p. 503]{G&S1982-GAIPOR}.} However, they argue that this problem can be avoided. Their fix is to modify the definition of a definable probability by restricting it to the empirical portion of the language. Because the domain and interpretations of the non-logical constants of $\pazocal{L}_A$ are fixed across models in $Mod(\pazocal{L}_A^+)$, we can further refine the class of statements over which probabilities on $\pazocal{L}_A^+$ are uniquely identified.

\begin{definition}
    Let $t_i$ be a term of $\pazocal{L}_A$ for all $i \in \omega$. We define the class of \textit{elementary empirical formulas} by recursion as follows:
    \begin{compactitem}
        \item[-] For all empirical predicate symbols $P \in \pazocal{L}_A^+ \setminus \pazocal{L}_A$, $P(t_1, \ldots, t_n)$ is an elementary empirical formula.
        \item[-] For all empirical function symbols $f \in \pazocal{L}_A^+ \setminus \pazocal{L}_A$, $f(t_1, \ldots, t_n)$ is an elementary empirical formula.
        \item[-] If $\varphi$ is an elementary empirical formula, then so is $\neg \varphi$.
        \item[-] If $\varphi$ and $\psi$ are elementary empirical formulas, then so is $\varphi \to \psi$. 
        \end{compactitem}
\end{definition}

\noindent
Since the connectives $\wedge$ and $\vee$ can be defined in terms of $\neg$ and $\to$, then it follows that the elementary empirical formulas are precisely the sentential combinations of the atomic empirical formulas. The elementary empirical \textit{sentences} can then be defined as the class of elementary empirical formulas containing no free variables. Using this definition, Gaifman and Snir show the following:

\begin{proposition}[\cite{G&S1982-GAIPOR}]
    Every Gaifman-Snir probability function on $\pazocal{L}_A^+$ is determined uniquely by its values on the elementary empirical sentences.
\end{proposition}

\noindent
Gaifman and Snir then take this proposition to warrant the following definition:

\begin{definition}
    A Gaifman-Snir probability function $P$ over $\pazocal{L}_A^+$ is \textit{definable} if its restriction to elementary empirical sentences is arithmetically definable.
\end{definition}

\noindent
The idea is that since probabilities are uniquely determined by their values on elementary sentences, this definition suggests that it is enough for our priors to be definable if they are definable on this restricted class.

While Gaifman and Snir's suggestion circumvents the problem raised by theorem 1, it seems to get rid of the need for having the arithmetic language around in the first place! Following this definition implies that our priors are only defined on the elementary empirical sentences. Of course, by definition, these don't contain any symbols from the language of arithmetic other than the numerals. And again, Theorem 1 assures us that there is no way to extend such a prior to all the sentences of $\pazocal{L}_A^+$ while preserving its definability.\footnote{It is well-known that $\mathsf{PA}$ can define a truth-predicate for the class of $\Sigma_n$-sentences for any $n \in \omega$ (see \cite{kaye1991models} section 9.3). In this way, one could extend the Gaifman and Snir notion of definability to a particular complexity class of arithmetic sentences and demonstrate the existence of priors of (at least) $\Sigma_{n+2}$ complexity. While this would admit arithmetically definable priors over \textit{fragments} of an arithmetic language, computable priors would only be definable over the bounded fragment of $\pazocal{L}_A$. This draconian restriction flies in the face of maximizing the complexity tradeoff and hence we do not see this as a live option.} Definable priors in the sense of definition 3 then can only represent our credences over non-arithmetic sentences.\footnote{These priors are essentially only defined over the weak logical language $\pazocal{L}_\infty$ of \cite{Carnap1950-CARLFO}.} This is rather unsatisfying since our original motivation was to increase the expressivity of our language so as to have priors defined over more complex sentences.

In summary, we examined Gaifman and Snir's framework with an eye towards satisfying Putnam's conditions and learned the negative lesson that (G1), (G2), and (G3) together yield an impossibility result for arithmetically definable probability functions. More optimistically, an impossibility result clearly lays out the options for how to move forward. The remainder of this paper will explore some options for arithmetically definable probability functions. 

Prima facie, the options seem to be to relax (G1), relax (G2), or relax (G3). However, this assumes that the conditions are independent so that relaxing any one of them is a distinct option. It will turn out that this is not the case here. In the next section, we will show that the Gaifman condition, (G3), is actually entailed by a fixed domain assumption we made when setting up (G1) and (G2) at the beginning. This shows that we must reformulate (G1) and (G2) without reference to a fixed domain to make them independent of (G3). This will leave us with a clearer set of options for how to proceed.

\section{Restrictive Assumptions}

In this section we explore the relationship between the Gaifman condition and an assumption already implicit in (G1) and (G2), the fixed domain assumption. By ``fixed domain assumption'', we mean the framework decision to let the model space used to define probability functions contain only models sharing a common domain. While Gaifman and Snir fix the standard numbers as the common domain underlying their model space, \cite{Paris2011-PARPIL-3} more generally assume a countable set of constants which they take to be the common domain. In what follows, we formalize the notion of a domain assumption as the condition that the domain of every model in the model space used to define probability functions consists only of a countable set of constants. Since the standard numbers and the numerals stand in a one to one correspondence, there is no formal difference in treating $\mathbb{N}$ as the structure $\langle \underline{\omega}, \underline{0}, S. +, \times \rangle$ where $\underline{\omega}$ is the set of numerals and we understand every numeral to be interpreted as itself.

Recall that (G1) and (G2) were defined in terms of the logical satisfaction relation between the set of models $Mod(\pazocal{L}_A^+)$ and elements of $\pazocal{SL}_A^+$. The fixed domain assumptions are therefore present in our definition of probability functions through our definition of the space $Mod(\pazocal{L}_A^+)$. In this section, we demonstrate that these assumptions alone are sufficient to guarantee the Gaifman condition, i.e., (G1) and (G2) imply (G3). Then we will show that any probability function over $\mathsf{PA}^+$ (with or without fixed domain assumptions) satisfying (G3) cannot be arithmetically definable. Since the domain assumption implies (G3), this also serves to show that the fixed domain assumption alone is sufficient to block arithmetically definable priors.

\subsection{Domain Assumptions}

Before we show how the fixed domain assumption leads to difficulties with arithmetical definability, we would like to begin more charitably by first recognizing how reasonable of an assumption it is. The fixed domain assumption is an extremely natural one to make when setting up an inductive logic. 

The choices we make in setting up an inductive logic are guided by how we intend to apply it. Many proposals in inductive logic have been guided by the intuition that the inductive logic is meant to apply to a countable sequence of individuals, each of which satisfies or does not satisfy some qualitative predicates or relations. This intuitive image motivates formalizing observations and predictions as formulae without quantifiers. A paradigm example of an evidential statement on this view is that the fiftieth observed individual is a white swan, formalized perhaps as $S(c_{50})\land W(c_{50})$. Evidential statements that require quantification, such as a statement about the observed rate of acceleration between two celestial bodies at a particular time, are suppressed by this picture. 

The purpose of a universally quantified sentence is to talk about all individuals. So, if we understand all individuals as all individuals of a countable sequence, it is very obvious why we should make the fixed domain assumption. To say that all swans are white is to say that the first swan is white, the second swan is white, and so on. Historically, taking the context of application to be a countable sequence originates in the frequentist interpretation of probability. But there are independent reasons to like this view. Even if we are open to the possibility that there are uncountably many individuals, we may still think of those that we observe as forming a countable sequence and require that the probabilities of all our sentences, including the quantified ones, should be settled at one or zero in the limit of these observations. This can be motivated in an empiricist, instrumental, or scientifically practicalist spirit. In an empiricist spirit, we may think that as a matter of semantics there should be no uncertainty once all of the observations in the countable sequence are in. In an instrumental spirit, we may be open to the existence of uncountably many individuals but insist that the purpose of inductive logic is only to help us reason about the countably many things we will actually come across. In a scientifically practicalist spirit, we may point out that it is unclear how to reason about the limit properties of statistical procedures when we do not help ourselves to definitions of limits in terms of countable sequences.

\subsection{Arithmetic Definability and the Gaifman Condition}

The fixed domain assumptions provide a natural motivation for the Gaifman condition as a way of characterizing the behavior of a probability function on quantified sentences. Since the domain is exhausted by the interpretation of the constants, the probability of an existentially quantified sentence $\exists x \psi(x)$ should be somehow determined by the instantiations $\psi(c_i)$ for $i \in \omega$. Because our language allows only finite sentences while there are infinitely many constants, taking the supremum of a sequence of longer disjunctions, as is done in the Gaifman condition, is a very sensible way of extending probabilities to quantified sentences. In fact, as the following proposition demonstrates, the Gaifman condition is not only well-motivated by these domain assumptions, but is logically entailed by them.

\begin{proposition}\label{prop:G1 and G2 imply G3}
    Let $\pazocal{L}$ be a language with countably many constants $c_i$ for all $i \in \omega$, and finitely many predicate and function symbols. Let $Mod(\pazocal{L})$ be the set of $\pazocal{L}$-structures with domain $\{c_i \mid i \in \omega\}$. Then any function $P: \pazocal{SL} \to [0,1]$ satisfying conditions (G1) and (G2) with respect to $Mod(\pazocal{L}) \vDash$ satisfies the Gaifman condition (G3).
\end{proposition}

\noindent
We first introduce a bit of machinery for this proof.
Let $\pazocal{B} :=\{\llbracket \varphi \rrbracket \mid \varphi \in \pazocal{SL}\}$ where 
\[\llbracket \varphi \rrbracket := \{ \pazocal{M} \in Mod(\pazocal{L}) \mid \pazocal{M} \vDash \varphi\}.\] 
To prove the above proposition, we will exploit a relationship between probability functions $P$ over sentences of our language and probability measures $\mu$ on the least $\sigma$-algebra containing $\pazocal{B}$, denoted $\sigma(\pazocal{B})$.

\begin{lemma}[\cite{G&S1982-GAIPOR}]\label{lem:Measure on Mod Space}
    Let $P: \pazocal{SL} \to [0,1]$ be a function satisfying (G1) and (G2). Then there is a countably additive measure $\mu_{P}$ on $\sigma(\pazocal{B})$ such that for all $\varphi \in \pazocal{SL}$, $P(\varphi) = \mu_{P}(\llbracket \varphi \rrbracket)$.
\end{lemma}

\begin{proof}[Proof of Proposition 3]
Let $P: \pazocal{SL} \to [0,1]$ be a function satisfying (G1) and (G2) with respect to $Mod(\pazocal{L}) \vDash$. Then by lemma 2.1 there is a probability measure $\mu_{P} : \sigma(\pazocal{B}) \to [0,1]$ such that $P(\varphi) = \mu_{P}(\llbracket \varphi \rrbracket)$ for all $\varphi \in \pazocal{SL}$. Now fix an arbitrary sentence of the form $\exists x \psi(x) \in \pazocal{SL}$.
Then we can infer that
   \begin{align*}
        P(\exists x \psi(x)) &=  \mu_{P}\big(\llbracket \exists x \psi(x)\rrbracket\big) \\
        &= \mu_{P}\big(\{\pazocal{M} \in Mod(\pazocal{L}) \mid \pazocal{M} \vDash \exists x \psi(x)\}\big) \\
        &= \mu_{P}\big(\{\pazocal{M} \in Mod(\pazocal{L}) \mid \pazocal{M} \vDash \psi(c_i) \text{ for some } i \in \omega\}\big) \\
        &= \mu_{P}\big(\bigcup_{i \in \omega} \llbracket \psi(c_i)\rrbracket\big) \\
        &= \text{sup} \big\{ \mu_{P}\big(\llbracket \bigvee_{i =1}^n \psi(c_i)\rrbracket\big) \mid n \in \omega\big\} \\
        &= \text{sup}\{P\big(\bigvee_{i =1}^n \psi(c_i)\big) \mid n \in \omega\},
    \end{align*}
    where the third equality relies on the fixed domain assumption.
\end{proof}

While it has been commonplace in defining probabilities over logical languages to both fix the domains of the structures in the model space as well as impose the Gaifman condition, Proposition \ref{prop:G1 and G2 imply G3} demonstrates that the domain assumptions alone suffice to ensure our probability functions satisfy the Gaifman condition. 

As we noted above, however, when we expand our languages to include that of arithmetic, these domain assumptions got in the way of our ability to find accessible priors. The following proposition demonstrates that even if we were to drop the domain assumptions, the Gaifman condition alone is sufficient to cause us problems in finding arithmetically definable priors. Let $V_\mathbb{N}$ be the function which assigns 1 to the sentences for which $\mathbb{N} \vDash \varphi$ and 0 otherwise.

\begin{theorem}
    Let $Mod(\pazocal{L}_A)$ be the set of models of $\mathsf{PA}$, i.e., the models of $\mathsf{PA}$ without any domain restrictions.
    Let $P : \pazocal{SL}_A \to [0,1]$ be a real-valued function which satisfies
    \begin{compactitem}
        \item[(i)] If $Mod(\pazocal{L}_A) \vDash \varphi$, then $P(\varphi)=1$
        \item[(ii)] If $Mod(\pazocal{L}_A) \vDash \neg (\varphi \wedge \psi)$, then $P(\varphi \vee \psi) = P(\varphi) \cdot P(\psi)$
    \end{compactitem}
    as well as the Gaifman condition (G3).
    Then for all $\varphi \in \pazocal{SL}_A$, $P(\varphi) = V_{\mathbb{N}}(\varphi)$.
\end{theorem}

\begin{proof}
    We prove the above by induction on the arithmetic hierarchy. In the case where $\varphi \in \Sigma_1$, $V_\mathbb{N}(\varphi) = P(\varphi)$ follows from the $\Sigma_1$-completeness of $\mathsf{PA}$ and (G1). 

    Next suppose $\varphi \in \Pi_n$; in particular, let $\varphi = \neg \psi$ with $\psi \in \Sigma_n$.
    By the induction hypothesis, $P(\psi)=V_\mathbb{N}(\psi).$
    Then
    \[P(\neg \psi) = 1 - P(\psi) = 1 - V_\mathbb{N}(\psi) = V_\mathbb{N}(\neg\psi),\]
    where the first equality follows from (ii).
    
    Finally suppose $\varphi \in \Sigma_{n+1}$. Suppose that $\varphi := \exists x \psi(x)$ for $\psi(x) \in \Pi_n$ (the more general case with a block of quantifiers follows similarly).
    In the case where $V_\mathbb{N}(\varphi) = 1$, we know that there is some $k \in \omega$ such that $V_\mathbb{N}(\psi(\underline{k})) = 1$. By our induction hypothesis and (G3) it follows that 
    \[P\big(\exists x \psi(x)\big) = \sup \big\{ P\big(\bigvee_{i\leq n}\psi(\underline{i})\big) \mid n \in \omega\big\} = 1.\]
    In the case where $V_\mathbb{N}(\varphi) = 0$, we know that there is no $k \in \omega$ such that $V_\mathbb{N}(\psi(\underline{k})) = 1$. Hence, again, our induction hypothesis along with (G3) gives us
     \[P\big(\exists x \psi(x)\big) = \sup \big\{ P\big(\bigvee_{i\leq n}\psi(\underline{i})\big) \mid n \in \omega\big\} = 0.\]
\end{proof}

Now observe that in the previous theorem we were only concerned with probabilities defined over $\pazocal{SL}_A$. The following corollary demonstrates that this is sufficient to block probabilities over the sentences of an extended language $\pazocal{SL}_A^+$ from being arithmetically definable. 

\begin{corollary}
    Let $P : \pazocal{SL}_A^+ \to [0,1]$ be a probability function as described in Theorem 3. Then $P$ is not arithmetically definable.
\end{corollary}

\begin{proof}
    Supposing towards a contradiction that $P$ is arithmetically definable, then there is some $f_{P}: \omega \times \omega \to \mathbb{Q}$ which approximates $\ulcorner P\urcorner$. Then the set 
    \[\{\ulcorner \varphi\urcorner \mid \text{Sent}(\ulcorner \varphi \urcorner) \wedge \forall n (|1-f_{P}(\ulcorner\varphi\urcorner, n)| \leq 2^{-n})\}\]
    is arithmetically definable since $f_{P}$ is.\footnote{Recall that Sent($x$) is a formula which holds iff $x$ is a Gödel code for a \textit{sentence} in the language of arithmetic. In this way, $\mathsf{PA}$ is able to isolate the sentences in $\pazocal{SL}_A$ from those in $\pazocal{SL}_A^+$.} By theorem 2.1, this is a truth predicate for the standard model $\mathbb{N}$ of $\mathsf{PA}$ which is in contradiction with Tarski's undefinability of truth theorem.
\end{proof}

This corollary tells us that the Gaifman condition alone is problematic when trying to construct `accessible' priors. Even after dropping the fixed domain assumptions, the Gaifman condition still causes definable priors to run into Tarski's undefinability of truth theorem. Moreover, it demonstrates that the Gaifman condition and the framework decision of restricting the domains of the structures in the model space are very closely related. As we have already seen, the Gaifman condition follows from the domain assumptions. While the Gaifman condition cannot influence the assumptions we make behind the $\vDash$ relation we use to define the probability function in the first place, theorem 2.1 demonstrates that, in the context of arithmetic, the Gaifman condition forces the conditional statement in condition (i) into a biconditional $\mathbb{N} \vDash \varphi \Leftrightarrow P(\varphi)=1$. 
Since the standard model is the only model whose domain is exhausted by interpretations of the terms in $\pazocal{L}_A$, then the Gaifman condition both trivializes the probability function over this language (as it collapses into truth in the standard model $\mathbb{N}$) and thereby assures that it is definable nowhere in the arithmetic complexity hierarchy.

\section{Weakening the Assumptions}

In this section, we explore possibilities for arithmetically definable probabilities through relaxing the Gaifman condition. In the previous section, we saw that the (G1) and (G2) conditions on probability functions entail the Gaifman condition. So, we must also relax these requirements from conditions expressed in terms of logical satisfaction to conditions expressed in terms of provability. Thus we have the following definition:

\begin{definition}\label{def: deductively closed probability function}
    A \textit{deductively-closed probability function} on $\mathsf{PA}^+_A$ is a function $P_{\vdash}: \pazocal{SL}^+_A \to [0,1]$ that satisfies the following axioms.
    \begin{enumerate}[label=(D\arabic*)]
        \item If $\mathsf{PA}^+ \vdash \varphi$ then $P_\vdash (\varphi)=1$.
        \item If $\mathsf{PA}^+ \vdash \neg(\varphi \land \psi)$ then $P_\vdash (\varphi \lor \psi) = P_\vdash (\varphi) + P_\vdash (\psi)$.
    \end{enumerate}
\end{definition}

Naturally this definition generalizes to any theory $T$, by replacing $\mathsf{PA}^+$ with $T$ in conditions (D1) and (D2).
In what follows we will only be concerned with $\mathsf{PA}^+$, so we will generally write ``deductively-closed probability function'' in place of the more cumbersome ``deductively-closed probability function on $\mathsf{PA}^+$''.
Note the similarity to our original (G1) and (G2). These conditions do the intended work, but without a fixed domain.

\subsection{A Limit Computable Prior}

We now demonstrate that there are arithmetically definable \textit{deductively-closed} probabilities over arithmetic languages. Hence the weakening of the conditions which define probability functions saves us from immediately running into issues with Tarski's undefinability of truth theorem. The construction is based on an observation of Feferman's (\cite{feferman1960}) that the theory $\mathsf{PA}$ is capable of defining complete consistent extensions of itself.\footnote{Of course, it is unable to prove that the set defined is indeed a complete consistent extension of $\mathsf{PA}$. But this is true in the standard model.} The idea we chase is simply to take such a definable complete extension and note that its characteristic function is a probability function.

As Feferman shows,\footnote{See \cite[pp. 46-47]{feferman1960}.} the Lindenbaum construction for generating maximally consistent sets of sentences from definable consistent sets of sentences can be carried out in $\mathsf{PA}$.\footnote{See also \cite{lindstrom1997}.} Since we know that the set of theorems of $\mathsf{PA}$ is $\Sigma_1$-definable, we briefly sketch how to define a complete consistent extension of $\mathsf{PA}^+$ in $\mathsf{PA}$.\footnote{For details see \cite{Chao2016GdelsSI} section 3.} Let $\nu : \omega \to \omega$ be a computable enumeration of the sentences of $\pazocal{L}_A^+$. Then the following set can be constructed in $\mathsf{PA}$:
\begin{align*}
    \Gamma_0 &:= \{\ulcorner \varphi \urcorner \mid \mathsf{PA}^+ \vdash \varphi\} \\
    \Gamma_{n+1} &:=
    \begin{cases}
        \Gamma_n \cup \{\ulcorner\varphi \urcorner \} & \text{ if } \Gamma_n \nvdash \neg \varphi \text{ where } \nu(n+1) = \ulcorner \varphi \urcorner \\
        \Gamma_n \cup \{\ulcorner\neg \varphi \urcorner\} & \text{ if } \Gamma_n \vdash \neg \varphi \text{ where } \nu(n+1) = \ulcorner \varphi \urcorner
    \end{cases}\\
    \Gamma &:= \bigcup_{n \in \omega} \Gamma_n.
\end{align*}

\noindent
The set $\Gamma$ is $\Delta_2$-definable.\footnote{See \cite{Chao2016GdelsSI} Theorem 6.} It is important to note here that $\Delta_2$ is the best we can do.

\begin{theorem}
    There are no $\Sigma_1$, or $\Pi_1$- definable complete consistent extensions of $\mathsf{PA}^+$.
\end{theorem}

\begin{proof}
    That there are no $\Sigma_1$-definable complete extensions follows immediately from Gödel's first incompleteness theorem and the $\Sigma_1$-completeness of $\mathsf{PA}$. Now suppose towards a contradiction there is a $\Pi_1$-definable complete extension $\Gamma$ of $\mathsf{PA}^+$. Then its complement set of sentences, 
    \[\Gamma' =\{\ulcorner\varphi \urcorner \mid \text{Sent}(\ulcorner \varphi \urcorner) \wedge \ulcorner\varphi\urcorner \notin \Gamma\}\]
    is $\Sigma_1$-definable. But then, by double negation elimination, the set
    \[X =\{\ulcorner\varphi \urcorner \mid \exists \ulcorner\psi\urcorner \in \Gamma' \; \text{Prf}_{\mathsf{PA}}(\ulcorner \varphi \leftrightarrow \neg \psi \urcorner)\} = \Gamma\]
    where $X$ is $\Sigma_1$ and hence so is $\Gamma$. But this is a contradiction since we have already established that there are no $\Sigma_1$-definable completions of $\mathsf{PA}$.
\end{proof}

We can then use this $\Delta_2$-definable complete extension of $\mathsf{PA}^+$, $\Gamma$, to define a probability function $P_\Gamma : \pazocal{SL}_A^+ \to [0,1]$ by letting
\begin{align*}
    P_\Gamma(\varphi) &:= 
    \begin{cases}
        1 & \text{ if } \ulcorner\varphi\urcorner \in \Gamma \\
        0 & \text{ if } \ulcorner \varphi \urcorner \notin \Gamma
    \end{cases}
\end{align*}
for all $\varphi \in \pazocal{SL}_A^+$.
Clearly, $P_\Gamma$ is $\Delta_2$-definable as it is the characteristic function of a $\Delta_2$-definable set. It is also easily seen that $P_\Gamma$ is a deductively-closed probability function as in definition 4.

As a prior, $P_\Gamma$ is rather unsatisfying. It can easily be seen that conditionalization does not accomplish anything over 0-1 probability functions. 
This can be improved upon. Suppose we have $E_1, \ldots, E_k$ empirical predicates in our language $\pazocal{L}_A^+$ for $k \geq 1$. Then we can permute the base case of our Lindenbaum construction above in order to obtain different probability functions. For any sentence of the form $\bigwedge_{l \leq k} \pm E_l(\underline{n_l})$ where $\pm E_i(\underline{n})$ denotes either $E_l(\underline{n})$ or $\neg E_l(\underline{n})$, the set 
\[\{\ulcorner \varphi\urcorner \mid \mathsf{PA}^+ \vdash \varphi\} \cup \{\ulcorner \bigwedge_{l \leq k} \pm E_l(\underline{n_l})\urcorner\}\] 
is also $\Sigma_1$-definable. Now let $\langle \psi_i  \mid {i \leq n}\rangle$ be a finite sequence of consistent sentences of the form $\bigwedge_{l \leq k} \pm E_l(\underline{n_l})$.
Now let $\Theta$ be the limit of the Lindenbaum construction over the base set
\[\{\ulcorner \varphi\urcorner \mid \mathsf{PA}^+ \vdash \varphi\} \cup \{\psi_i \mid i \leq n\}.\]
Then again, $\Theta$ is a $\Delta_2$-definable complete extension of $\mathsf{PA}^+$ and hence the characteristic function $P_\Theta$ is a $\Delta_2$-definable, deductively-closed probability function over $\pazocal{SL}_A^+$. 

Now for any series of sequences $\langle \psi^0_i \mid i \leq n_0\rangle, \ldots, \langle \psi^m_i \mid i \leq n_m\rangle$ as defined above, we can obtain a series of deductively-closed probability functions $P_{\Theta_0}, \ldots, P_{\Theta_m}$ where $P_{\Theta_j}$ is defined as characteristic function of the limit of the Lindenbaum construction over
\[\{\ulcorner \varphi\urcorner \mid \mathsf{PA}^+ \vdash \varphi\} \cup \{\psi^j_i \mid i\leq n_j\}\]
for all $0 \leq j \leq m$.

We can now take a mixture of these probability functions to obtain a non-trivial arithmetically definable probability function. To do this, fix a sequence $\langle q_j \mid j \leq m \rangle$ of rational numbers such that $0 < q_j < 1$ for all $j \leq m$ and $\sum_{j \leq m} q_j =1$. Then we can define a probability function $P : \pazocal{SL}_A^+ \to [0, 1]$ by letting
\[P (\varphi) := \sum_{j \leq m} q_j \cdot P_{\Theta_j}(\varphi)\]
for all $\varphi \in \pazocal{SL}_A^+$. Since every $P_{\Theta_j}$ for $0 \leq j \leq m$ is $\Delta_2$-definable and we can find computable sequences $\langle q_j \mid j \leq m\rangle$ of rational numbers summing to 1,\footnote{Take, for example, $q_j = \frac{1}{m}$ for all $0 \leq j \leq m$.} then the mixture $\sum_{j \leq m } q_j \cdot P_{\Theta_j}(\varphi)$ is also $\Delta_2$-definable. By choosing a sequence of rationals composed not merely of 1's and 0's, we obtain a non-trivial $\Delta_2$-definable probability function over $\mathsf{PA}^+$ on which conditionalization can be shown to in fact change the prior. Hence there are $\Delta_2$-definable priors which allow for non-trivial Bayesian updating.

This establishes the following:

\begin{theorem}
    There are $\Delta_2$-definable deductively-closed probability functions on $\pazocal{SL}_A$.
\end{theorem}

In fact, our construction method demonstrates that there is at least a countable infinity of non-trivial -- in the sense that conditionalization can change the prior -- $\Delta_2$-definable probability functions on $\pazocal{SL}_A^+$ satisfying (D1) and (D2).

% [While this is the best news you're gunna get in this paper, it's still not really good news... Discuss limit computability interpretation of $\Delta^0_2$, with epistemic interpretation. Until the limit, you have no useful information. Even c.e. or co-c.e. has \textit{something}.]

This is progress. We have an arithmetically definable probability function that, by satisfying the (D1) and (D2) conditions, can track mathematical reasoning in our scientific theories. We can treat this as satisfaction of Putnam's first condition. But is a $\Delta_2$ probability function `accessible' in the sense of Putnam's second condition? We think not. A $\Delta_2$ probability function can be intuitively glossed as a `limit computable' function. There is an effective procedure that, given any sentence, can yield a sequence of outputs whose limit is the desired probability. So we have some traction on the probability function. However, there is no guarantee at any finite step that our result is close to the limit to any degree of approximation nor even whether it is higher or lower. Contrast this with effective procedures for $\Sigma_1$ functions, whose outputs we at least know approach the value from below so that we can determinately say ``the probability is at least...'', and similarly the $\Pi_1$ functions in the opposite direction. In the $\Delta_2$ case, we have no such guarantees and so only have a grasp on the quantity `in the limit'. For finite beings like us, a miss here is as good as a mile. If the function can only be computed in the limit, it is for our purposes as useless as a function that cannot be arithmetically defined in the first place. But $\Delta_2$ puts us close. If only we could go any lower in the computability hierarchy we would have determinate information about our probability function at finite stages of computation. In the next section, we will learn the bitter lesson that $\Delta_2$ is the best we can do with this approach.

You might think that another problem with a $\Delta_2$ probability function is that what is limit computable is the prior probability and not a posterior or empirical quantity. What we get in the infinite limit of computation is not our final epistemic aim but only a starting point for inquiry. This seems quite useless if our goal in using inductive logic is to go beyond our starting point and make inductive inferences. What we really want is traction not on the prior probabilities but on the conditional probabilities. Do we have that? Bayes' rule tells us that every conditional probability is the ratio of two prior probabilities, each $\Delta_2$ and where we assume both are positive. Recalling that a $\Delta_2$ function is $\Delta_1$ relative to the Halting problem $\varnothing'$, it is easy to see that just as ratios of $\Delta_1$ quantities are $\Delta_1$ when well defined, so too ratios of $\Delta_2$ quantities are $\Delta_2$ when well defined. So we do have limit computability of the conditional probabilities as well. However, this is ultimately not satisfying.

Procedures that only have long-run guarantees are roundly criticized, often with Keynes' famous quip that ``in the long run, we are all dead'', but this is too quick. After all, many statistical procedures can be understood as zeroing in on an epistemic target only in the limit. Applications of large sample theory are ubiquitous in experimental designs and explicitly rely on treating asymptotic results as reliable for reasoning from finite (but large) datasets. However, it would be misleading to think that relying on a $\Delta_2$ probability measure is equally defensible. Large sample theory gives certain guarantees on rates of convergence as the number of samples increases. Scientists with good sense will only trust asymptotically valid inferences on finite data if the rate of convergence is reasonable. We have no analogous guarantees on the rate at which fluctuations from the limit eventually diminish with computation time. Without such guarantees, $\Delta_2$ probability measures are wholly unsatisfying.

\subsection{Limits \textit{of} Computable Priors}

In this section we show that, unfortunately, the $\Delta_2$ upper bound is tight: there are no $\Sigma_1$- or $\Pi_1$-definable deductively-closed probability functions.
To begin we prove an analogue of Lemma \ref{lem:Measure on Mod Space} for deductively-closed probability functions.
We let 
\[[\pazocal{M}]_\equiv = \{\pazocal{N} \mid \pazocal{N} \equiv \pazocal{M}\}\]
denote equivalence classes of structures up to elementary equivalence, and let 
\[\mathcal{M}_\pazocal{L} = Mod(\mathsf{PA}^+)/_\equiv = \{[\pazocal{M}]_\equiv \mid \pazocal{M} \vDash \mathsf{PA}^+\}\]
be the quotient space of models of $\mathsf{PA}^+$ up to elementary equivalence.
It is important to notice here that we are not making any domain assumptions on this model space.

\begin{lemma}\label{lem:DC Measure on Mod Space}
    \begin{enumerate}
        \item Every Borel probability measure on $\mathcal{M}_\pazocal{L}$ determines a deductively-closed probability function on $\pazocal{SL}^+_A$.
        \item Every deductively-closed probability function on $\pazocal{SL}_A^+$ determines a Borel probability measure on $\mathcal{M}_\pazocal{L}$.
    \end{enumerate}
\end{lemma}

\begin{proof}
    Let $\pazocal{B}:=\{\llbracket\varphi\rrbracket: \varphi \in \pazocal{SL}^+_A\}$, where $\llbracket\varphi\rrbracket:=\{[\pazocal{M}]_\equiv: M \in Mod(\mathsf{PA}^+) \land M \vDash \varphi\}.$

    For (1), let $\mu: \sigma(\pazocal{B}) \to [0,1]$ be a probability measure and define $P_\mu(\varphi) := \mu(\llbracket\varphi\rrbracket)$ for all $\varphi \in \pazocal{SL}_A.$
    Then if $\mathsf{PA}^+ \vdash \varphi$ we have $P_\mu(\varphi)=1$.
    If $\mathsf{PA}^+ \vdash \neg(\varphi \land \psi)$ then $\llbracket\varphi\rrbracket \cap \llbracket\psi\rrbracket = \emptyset$, so 
    \[P_\mu(\varphi \lor \psi) = \mu(\llbracket\varphi\rrbracket \cup \llbracket\psi\rrbracket)=\mu(\llbracket\varphi\rrbracket)+\mu(\llbracket\psi\rrbracket)=P_\mu(\varphi) + P_\mu(\psi).\]

    For (2), let $P$ be a deductively-closed probability function and define $\mu_P(\llbracket\varphi\rrbracket):=P(\varphi).$ 
    It suffices to show that $\mu_P$ is finitely additive and continuous from above at $\emptyset$ on $\sigma(\pazocal{B})$ (\cite{ash2000probability}).

    First we show that $\mu_P$ is finitely additive. Let $\llbracket\varphi\rrbracket \cap \llbracket\psi\rrbracket=\emptyset.$ Then every model $\pazocal{M}$ satisfies $\pazocal{M} \vDash \neg(\varphi \land \psi)$, and so $\mathsf{PA}^+ \vdash \neg(\varphi \land \psi).$ Therefore $\mu_P(\llbracket\varphi\rrbracket \cup \llbracket\psi\rrbracket)=P(\varphi \lor \psi)=P(\varphi)+P(\psi)=\mu_P(\llbracket\varphi\rrbracket)+\mu_P(\llbracket\psi\rrbracket).$ 

    Second we show that $\mu_P$ is continuous from above at $\emptyset.$
    Suppose $\langle\llbracket\varphi_j\rrbracket \mid j \in J \rangle$ is a countable sequence of sets such that $\llbracket\varphi_{j+1}\rrbracket\subseteq \llbracket\varphi_j\rrbracket$ for all $j \in J$ and $\bigcap_{j\in J} \llbracket\varphi_j\rrbracket=\emptyset.$
    By the Compactness Theorem for first-order logic, since there is no model satisfying all of $\{\varphi_j \mid j \in J\}$ it must be the case that for some finite set $\{\varphi_1, \ldots, \varphi_n\}$ there is no model $\pazocal{M}$ satisfying $\pazocal{M} \vDash \bigwedge_{i \leq n} \varphi_i.$ 
    Therefore $\mathsf{PA}^+ \vdash \neg (\bigwedge_{i \leq n}\varphi_i).$ 
    Thus $P\left(\bigwedge_{i \leq n}\varphi_i \right)=0,$ hence $\mu_P\left(\big\llbracket\bigwedge_{i \leq n}\varphi_i\big\rrbracket\right)=0.$ 
    But since for any $m>n$ we have $\big\llbracket\bigwedge_{i \leq m}\varphi_i\big\rrbracket \subseteq \big\llbracket\bigwedge_{i \leq n}\varphi_i\big\rrbracket,$ it follows that $\lim_{n\to\infty} \mu_P\left(\big\llbracket\bigwedge_{i \leq n}\varphi_i\big\rrbracket\right)= 0$.

    Therefore $\mu_P$ is countably additive, and by the Carathéodory Extension Theorem (\cite{ash2000probability}) extends to a countably additive Borel probability measure on $\sigma(\pazocal{B}).$
\end{proof}

So, as with Gaifman-Snir probability functions, each deductively-closed probability function corresponds to a unique countably-additive probability measure on the space of models (up to elementary equivalence) of $\mathsf{PA}^+$.
Each equivalence class $[\pazocal{M}]_\equiv$ determines a complete, consistent extension of $\mathsf{PA}^+$, namely $Th(\pazocal{M})$ for any (and hence all) $\pazocal{M} \in [\pazocal{M}]_\equiv$.
So another way to state Lemma \ref{lem:DC Measure on Mod Space} is: the support of a deductively-closed probability function contains only completions of $\mathsf{PA}^+$.
This is the key insight required for the main theorem of this section, and we make it precise as follows.

%Since we want to analyze the arithmetical complexity of deductively-closed probability functions it will be convenient to represent sets of formulae in Cantor space, $2^\omega$.
%As is common in computability theory we may identify a set $A \subseteq \omega$ of natural numbers with its characteristic sequence $\chi_A \in %Letting $\chi_A(n)$ denote the $n^{\text{th}}$ digit of the sequence, we define $\chi_A$ as
%\[n \in A \iff \chi_A(n)=1.\]
%We will abuse notation and write, e.g., $A$ for both the set and its characteristic sequence.
%Since we have fixed a Gödel numbering we may also identify sets of (Gödel numbers of) formulae of $\pazocal{L}_A$ with sequences by the correspondence
%\[\ulcorner \varphi \urcorner \in B \iff B(n)=1.\]

\begin{theorem}\label{thm:No DC Probs Below Delta_2}
    There are no $\Sigma_1$-definable or $\Pi_1$-definable deductively-closed probability functions on $\pazocal{SL}_A$.
\end{theorem}

The proof of this theorem is a combination of two simple lemmas: a ``leftmost-path'' construction, and a property of deductively-closed probability functions.

% Should we skip the translation into Cantor space and just work with the measure on model space?
\begin{lemma}\label{lem:P-Computable Set}
    Let $P$ be a probability measure on $2^\omega$.
    Then there is a sequence $X \in 2^\omega$ such that for all $n$, $P(X \upharpoonright n) > 0$, and $X$ is $\Delta_1^P$-definable.
\end{lemma}

\begin{proof}
    Begin by checking whether $P([1])>0$ or $P([0])>0$ (equivalently, one can check whether $P([1])<1$ or $P([0])<1$).
    These conditions are c.e. in $P$.
    Moreover at least one of these conditions must hold. 
    Let $X_1$ be the least number satisfying either $P([X_1])>0$ or $P([1-X_1])<1$.
    Recursively define $X_{n+1}$ by checking $P([X_n0])>0$ and $P([X_n1])>0$.
    As before at least one must hold; let $X_{n+1}$ be the length-lexicrographically least string satisfying this condition.

    By construction, for all $n$, $P(X \upharpoonright n)>0$.
    Since the value of each $X_n$ is determined at some finite time, $X = \lim_n X_n$ is computable from $P$, i.e., is $\Delta_1^P$-definable.
\end{proof}

\begin{lemma}\label{lem:Support is Completions}
    Let $P_\vdash$ be a deductively-closed probability function on $\mathsf{PA}^+$.
    If $X \in 2^\omega$ satisfies $P_\vdash^*(X \upharpoonright n)>0$ for all $n$, then $X$ is (the characteristic sequence of) a complete consistent extension of $\mathsf{PA}$.
\end{lemma}

\begin{proof}
    By Lemma \ref{lem:DC Measure on Mod Space}, there is a corresponding probability measure $\mu_{P_\vdash}$ on $\mathcal{M}_\pazocal{L}$.
    Let $X \in 2^\omega$ be such that $P_\vdash^*(X \upharpoonright n) > 0$ for all $n$.
    Fix the enumeration $\{\varphi_n\}_{n \in \omega}$ of $\pazocal{SL}_A$  such that $\ulcorner \varphi_m \urcorner = m$.
    Let $\psi_i = \varphi_i$ if $X_i=1$, and $\psi_i = \neg\varphi_i$ otherwise.
    Then we have
    \[\mu_{P_\vdash}\left(\bigcap_{i \leq n}\llbracket\psi_i\rrbracket \right)>0\]
    for all $n$. 
    Hence for all $n$ there is a model $\pazocal{M} \vDash \bigwedge_{i \leq n}\varphi_i$, and so by Compactness there is a model $\pazocal{M}^*$ such that
    \[\varphi_i \in Th(\pazocal{M}^*) \iff \pazocal{M}^* \vDash \varphi_i \iff X_i=1.\]
    Thus $X$ is the characteristic sequence of a complete, consistent extension of $\mathsf{PA}^+$.
\end{proof}

\begin{proof}[Proof of Theorem \ref{thm:No DC Probs Below Delta_2}]
    Suppose $P_\vdash$ is a $\Sigma_1$-definable deductively-closed probability function.
    Then its corresponding probability measure $P_\vdash^*$ on $2^\omega$ is also $\Sigma_1$-definable.
    By Lemma \ref{lem:P-Computable Set}, there is a sequence $X \in 2^\omega$ such that, for all $n$, $P_\vdash^*(X \upharpoonright n)>0$, and $X$ is $\Delta_1^{P^*_\vdash}$-definable, and hence is $\Sigma_1$-definable.
    But by Lemma \ref{lem:Support is Completions}, $X$ is the characteristic sequence of a complete consistent extension of $\mathsf{PA}$.
    But since there are no $\Sigma_1$ completions of $\mathsf{PA}$, this is impossible.
    Since there are no $\Pi_1$ completions of $\mathsf{PA}$, by a parallel argument we establish that there are no $\Pi_1$-definable deductively-closed probability functions on $\mathsf{PA}^+$.
\end{proof}

\section{Conclusion}

We began with two natural constraints on a useful inductive logic: the language should be sufficiently rich to express the kinds of hypotheses that interest scientists, and the probabilities assigned to sentences in this language ought to be, in some sense, accessible.
We glossed the first requirement as an assertion that a useful inductive logic should be written in a language at least as expressive as first-order arithmetic;
we glossed the second as the claim that the probability function should be computable.
We then showed that, following the natural axiomatization due to Gaifman and Snir, these two requirements are not jointly satisfiable.
There is no Gaifman-Snir probability function that is arithmetically definable, much less computable.
We identified the Gaifman condition, (G3), as the culprit.

So, we sought alternatives.
We showed that the standard domain assumptions, namely that the language is interpreted over the standard model $\mathbb{N}$, actually \textit{implies} the Gaifman condition in the presence of the other axioms.
Hence the domain assumptions must be jettisoned if we wish to define a computable probability function over sentences.
In the spirit of the original Gaifman-Snir axioms, we suggested a weaker set of axioms which relied on provability in $\mathsf{PA}$ rather than truth in $\mathbb{N}$.
We showed that this improves the situation considerably: there are $\Delta_2$-definable probability functions that satisfy our provability axioms.
But, unfortunately, we showed that this bound cannot be improved.

So we still have not secured a computable probability function over the sentences of $\mathsf{PA}$ that tracks their logical structure in a natural way.
In retrospect this is perhaps not too surprising.
The set of theorems of $\mathsf{PA}$ is not computable;
any probability function that assigns probability 1 to all theorems must have some means for determining theoremhood, and this cannot be a decidable procedure.
Indeed the requirement that the probability function track truth in $\mathbb{N}$ or provability in $\mathsf{PA}$ is a particular instance of the well-known problem of logical omniscience in Bayesian epistemology (\cite{Hacking1967-HACSMR}, \cite{Garber1983}, \cite{Pettigrew2021logical}).
So our results sharpen those discussions by demonstrating that a logically omniscient probability function on sentences of arithmetic is not only unrealistic, it cannot even be computed by a Turing machine.

Our results also indicate natural avenues for further research.
On one hand we are still unsure how to define a computable probability function that, in some sense, learns \textit{logically}.
Of course, one could define a much weaker kind of probability function that is completely ignorant of the inferential relations between sentences in its domain; but, intuitively, this kind of probability function obviates the motivation for using a logical language in the first place.
So what we need is a middle ground.
Taking inspiration from the literature on logical omniscience, one could try defining probabilities over ``impossible worlds'': sets of sentences that are not necessarily consistent, for example.
Rather than requiring that the probability function ``know'' all theorems of $\mathsf{PA}$ one could instead only require that $P(B \mid A, A \to B)=1$, i.e., that if an agent has learned that $A$ and learned that $A$ implies $B$, then they are capable of concluding that $B$.
Exploring this proposal and related intuitions may be fruitful.

On the other hand, real scientific hypotheses require much more expressive power than arithmetic.
Statements in physics, chemistry, mathematical biology, statistics, etc. require at least the ability to express concepts from real and complex analysis, including areas from differential geometry to measure theory.
So one might conclude that a \textit{truly} useful inductive logic ought to be defined on a language at least as expressive as \textit{second-order} arithmetic (see \cite{Simpson_2009} for more).
This line of research pulls in the opposite direction of the previous one: moving to an even richer language like second-order arithmetic will likely make it even more challenging to define an interesting kind of computable probability function.
Nonetheless, actual scientific practice calls for this extension, and we expect that it will lead to a deep and interesting area of research.

\printbibliography

\end{document}